\documentclass{amsart}
\usepackage{amsmath, amsthm, amssymb}
\usepackage{mathrsfs}
\usepackage{algorithm}
\usepackage{algorithmic}
\usepackage[all]{xy}
\usepackage{graphicx}
\usepackage{caption}
\usepackage{mathptmx}
\usepackage[margin=4.1cm]{geometry}
\usepackage{url}
\usepackage{combelow}
\usepackage{subfig}

\allowdisplaybreaks

\floatname{algorithm}{Procedure}

\theoremstyle{plain}
\newtheorem{thm}{Theorem}[section]

\newtheorem{lemma}[thm]{Lemma}

\newtheorem{coro}[thm]{Corollary}
\theoremstyle{definition}
\newtheorem{defn}[thm]{Definition}
\newtheorem{ex}[thm]{Example}
\theoremstyle{remark}
\newtheorem{rem}[thm]{Remark}

\newcommand{\thmref}[1]{Theorem~\ref{#1}}
\newcommand{\lemmaref}[1]{Lemma~\ref{#1}}
\newcommand{\cororef}[1]{Corollary~\ref{#1}}
\newcommand{\exref}[1]{Example~\ref{#1}}
\newcommand{\secref}[1]{Section~\ref{#1}}

\newcommand{\tabref}[1]{Table~\ref{#1}}
\newcommand{\figref}[1]{Figure~\ref{#1}}

\newcommand{\NN}{\mathbb{N}}

\newcommand{\RR}{\mathbb{R}}
\newcommand{\CC}{\mathbb{C}}

\newcommand{\PP}{\mathbb{P}}

\newcommand{\N}[1]{N_{#1}}
\newcommand{\Ns}[2]{N_{#1}#2}

\newcommand{\im}[1]{\operatorname{im} (#1)}
\newcommand{\codim}[1]{\operatorname{codim} (#1)}

\newcommand{\Gr}[2]{\operatorname{\mathbb{G}} (#1,#2)}

\newcommand{\adj}[1]{\operatorname{adj}(#1)}
\newcommand{\EDD}[1]{\operatorname{\epsilon}(#1)}

\renewcommand{\dim}[1]{\operatorname{dim} (#1)}
\renewcommand{\deg}[1]{\operatorname{deg} (#1)}

\begin{document}

\title{The numerical algebraic geometry of bottlenecks}
\address{Department of mathematics, KTH, 10044,
  Stockholm, Sweden}
\urladdr{https://people.kth.se/~daek/}
\author{David Eklund}
\email{daek@math.kth.se}
\keywords{Numerical algebraic geometry, systems of polynomials, triangulation of
  manifolds, reach of manifolds}
\subjclass[2010]{14Q20, 65D18}

\begin{abstract}
This is a computational study of bottlenecks on algebraic
varieties. The bottlenecks of a smooth variety $X \subseteq \CC^n$ are
the lines in $\CC^n$ which are normal to $X$ at two distinct
points. The main result is a numerical homotopy that can be used to
approximate all isolated bottlenecks. This homotopy has the optimal
number of paths under certain genericity assumptions. In the process
we prove bounds on the number of bottlenecks in terms of the Euclidean
distance degree. Applications include the optimization problem of
computing the distance between two real varieties. Also, computing
bottlenecks may be seen as part of the problem of computing the reach
of a smooth real variety and efficient methods to compute the reach
are still to be developed. Relations to triangulation of real
varieties and meshing algorithms used in computer graphics are
discussed in the paper. The resulting algorithms have been implemented
with Bertini \cite{bertini} and Macaulay2 \cite{M2}.
\end{abstract}

\maketitle

\section{Introduction}
Let $X, Y \subseteq \CC^n$ be smooth varieties. For $x \in X$ and $y
\in Y$ let $\Ns{x}{X} \subseteq \CC^n$ and $\Ns{y}{Y} \subseteq \CC^n$
denote the normal spaces at $x$ and $y$. Consider the incidence
correspondence
\[
I(X,Y) = \{(x,y) \in X\times Y: y \in \Ns{x}{X}, \; x \in
\Ns{y}{Y}\}.
\]
The goal is to set up an optimal numerical homotopy for the purpose of
approximating the isolated points of $I(X,Y)$ given equations defining
$X$ and $Y$ as well as the dimensions of $X$ and $Y$. In this context
optimality means that the number of homotopy paths is equal to the
number of solutions. In the sequel the isolated points of $I(X,Y)$
will be denoted $I(X,Y)_0$.

The main idea behind the homotopy is to first solve the following two
initial problems for general points $p,q \in \CC^n$:
\[
\begin{array}{cc}
\Sigma_1=\{x \in X: p \in \Ns{x}{X}\}, & \Sigma_2=\{y \in Y: q \in \Ns{y}{Y}\}.
\end{array}
\]
The set of start points for the homotopy is essentially the product
set $\Sigma_1 \times \Sigma_2$. The system is then deformed so that
$p$ approaches $y$ and $q$ approaches $x$, which yields points on the
incidence set $I(X,Y)$. One may, as we do below, let $p=q$.

For $(x,y) \in I(X,Y)$ with $x\neq y$ the line joining $x$ and $y$ is
normal to $X$ at $x$ and $Y$ at $y$. We will refer to these as
bottlenecks. These lines correspond to the nontrivial critical
points of the squared Euclidean distance function $\sum_{i=1}^n
(x_i-y_i)^2$, nontrivial meaning that $x\neq y$. Now, if $Y=\{p\}$ is
just a point, this reduces to a type of polar variety which will be
called the normal locus of $X$ with respect to $p$: $\{x \in X: p \in
N_xX\}$. The normal class, which is the cohomology class of the normal
locus for a general $p \in \CC^n$, has been studied in \cite{DHOST,P}
where its degree is called the Euclidean distance degree. See also
\cite{H} for a relation to the problem of finding real points on
algebraic varieties.

In \secref{sec:applications} we describe some applications of
bottlenecks to real geometry. \secref{sec:homomethods} and
\secref{sec:EDD} contain background material on numerical homotopy
methods and the Euclidean distance degree. In \thmref{thm:count} we
show that optimality holds for the homotopy if $I(X,Y)$ is smooth and
finite and $X$ and $Y$ satisfy some additional genericity
assumptions. This means that under these assumptions, $I(X,Y)$ is the
product of the Euclidean distance degrees of $X$ and $Y$. The homotopy
itself is formulated in \secref{sec:homotopy} and in
\thmref{thm:paths} we prove that the isolated points of $I(X,Y)$ are
included among the end points of the homotopy. In the process we prove
bounds on the number of bottlenecks in terms of the Euclidean distance
degrees of $X$ and $Y$, see \cororef{coro:bound1} and
\cororef{coro:bound2}. In \secref{sec:projections} we briefly consider
a variant of the problem where $X$ and $Y$ are projected to a lower
dimensional affine space. The motivation for this is the usefulness of
dimensionality reduction in various applications. In
\secref{sec:examples} and \secref{sec:complexity} we give some
illustrative examples and compare the method to a more naive approach
to solve the problem.

\subsection{Applications to real geometry} \label{sec:applications}
Let $X, Y \subseteq \RR^n$ be real algebraic sets and suppose that the
complexifications $X_{\CC}$ and $Y_{\CC}$ are smooth complex
varieties. The real points of the Euclidean normal space
$\Ns{x}{X_{\CC}}$ at a real point $x \in X$ is equal to the standard
normal space of the real submanifold $X \subseteq \RR^n$. The method
presented in this paper may be applied to the problem of computing the
distance between $X$ and $Y$. If we let $I(X,Y)$ be defined the same
way as for complex varieties, we have that $I(X,Y)$ consists of the
real points of $I(X_{\CC},Y_{\CC})$. If there is a point $(x,y) \in
X\times Y$ realizing the infimum of the distance between points of $X$
and $Y$, for example if $X$ or $Y$ is compact, then $(x,y) \in
I(X,Y)$. The same is true of the maximal distance if $X$ and $Y$ are
compact. Considering the case $X=Y$, the method may also be applied to
compute a lower bound for the distance between connected components of
$X$. As above, if there is a point $(x,y) \in X\times X$ that realizes
the minimal distance between points of any two different connected
components of $X$, for example if $X$ is compact and not connected,
then $(x,y) \in I(X,X)$.

The special case $X=Y$ is of particular interest. The meaning of the
isolated points of $I(X, X)$ relates to the \emph{condition number} or
\emph{reach} of $X$. This is roughly speaking the maximal size of an
$\epsilon$-neighborhood of $X$ that embeds smoothly in $\RR^n$. More
precisely, the reach of $X$ is the supremum of the set of numbers $r
\geq 0$ such that all points in $\RR^n$ at distance less than $r$ from
$X$ has a unique closest point on $X$. Assume that $X$ is compact and
$\dim{X}>0$, in which case it has positive and finite reach. The reach
is an important invariant for methods that seek to build a model of
$X$ by covering it with balls of the ambient space $\RR^n$ and forming
the \u{C}ech complex or Vietoris-Rips complex corresponding to the
balls. This has been proposed as a method to compute the homology of
$X$ \cite{NSW}, see also the papers \cite{BKSW,DEHH,HW} on persistent
homology in algebraic geometry. More generally, triangulation, meshing
and similar procedures are important problems in computer graphics
among other disciplines \cite{ACM, meshing, BG, BO, CL}. The reach is
typically part of the input to algorithms proposed to solve such
problems. Let $\rho$ be the minimal radius of curvature on $X$, where
the radius of curvature at a point $x \in X$ is the reciprocal of the
maximal curvature of a geodesic passing through $x$. Also, note that
$I(X_{\CC},X_{\CC})$ contains $X_{\CC}$ as an excess component in the
form of $X_{\CC} \times X_{\CC}$ intersected with the diagonal of
$\CC^n \times \CC^n$. The reach can be calculated in two stages, it is
the minimum of $\rho$ and $\frac{1}{2} \inf\{||x-y||:(x,y) \in I(X,X)
\setminus X\}$, see \cite{reachestimate}. If $I(X_{\CC},X_{\CC})
\setminus X_{\CC}$ is finite and $X_{\CC}$ satisfies the genericity
conditions explained in \secref{sec:EDD}, the latter infimum is a
minimum which may be computed effectively using the homotopy presented
in this paper.

\subsection{Related work and future developments}
There are some similarities between the method presented in this paper
and the $\CC^*$-action homotopy explored in the \cite{DRESW} but that
homotopy involves higher powers of the path variable which is not the
case for the homotopy described below. Based on numerical experiments
this seems to be of importance in practice even though we have not
carried out any such complexity analysis. Also, the intersection
method presented in \cite{SVW} is closely related to the present work.

As we are interested in applications to real geometry and the number
of real solutions to the problem is sometimes much smaller than the
number of complex solutions, it is worth considering a similar
approach to compute the real solutions directly. For example one may
approach this problem via real path tracking \cite{BS}.

We have chosen to consider the Euclidean normal bundles of $X$ and $Y$
for the sake of applications to real geometry. However, the same
technique is applicable for more general vector bundles on $X$ and $Y$
which should make possible other interesting applications.

\section{Numerical homotopies for isolated roots} \label{sec:homomethods}

In this section we give a brief summary of numerical homotopy methods,
see \cite{BHSW, SW} for details. These methods can be used to
numerically find solutions to polynomial systems. Moreover it is
possible to guarantee that all isolated solutions of the given system
have been found. In homotopy methods we set up a deformation from a
\emph{start system} to the \emph{target system}, which is the system
that we would like to solve. The start system typically has known
solutions and the idea is to track these points using a numerical
predictor/corrector method to solutions of the target system.

More in detail, consider $m$ polynomials $H=\{H_1,\dots,H_m\} \subset
\CC[x_1,\dots,x_m,t]$. Let $\mathcal{X} \subseteq \CC^m \times \CC$ be
the subscheme defined by the ideal $(H_1,\dots,H_m)$ and let $\pi:
\mathcal{X} \rightarrow \CC$ be the projection. Suppose that we are
interested in $X=\pi^{-1}(0)$, for instance we might want to
approximate the isolated points of $X$. An example of this situation
is that we are given $m$ polynomials $\{F_1,\dots,F_m\}$ which define
$X$ and let $H_i = F_i + tG_i$ for sufficiently general $G_i \in
\CC[x_1,\dots,x_m]$ with $\deg{F_i}=\deg{G_i}$.

Homotopy methods can be used to approximate isolated points of $X$ by
first approximating the isolated points of another fiber, say
$\pi^{-1}(1)$. Suppose for simplicity that $\pi^{-1}(1)$ is smooth and
finite. We then choose a path $\alpha:[0,1] \rightarrow \CC$ with
$\alpha(1)=1$ and $\alpha(0)=0$ and track the so-called solution
paths. The solution paths are maps $\beta:(0,1] \rightarrow
  \mathcal{X}$ which satisfy $\pi \circ \beta=\alpha$ and there is one
  for each isolated point of $\pi^{-1}(1)$, provided that $\alpha$ is
  general enough as to avoid a finite number of points in $\CC$. After
  projection to $\CC^m$ a solution path may also be viewed as a map
  $(0,1] \rightarrow \CC^m$, a perspective we will employ in the
    sequel. The solution paths satisfy an ODE known as the Davidenko
    equation and they can be tracked using an ODE solver. The isolated
    points of $\pi^{-1}(1)$ are known as start points in this context
    and they are the initial values for the ODE. If $\beta$ is a
    solution path that converges, the limit point is called the end
    point of the path. The tracking is simplest in the case of
    non-singular paths, that is when the Jacobian of $H$ with respect
    to $\{x_1,\dots,x_m\}$ has full rank at all points of the path,
    including the end point if the path is convergent. The end point
    is in this case a smooth isolated point of $X$. The next simplest
    case is a convergent solution path which is non-singular except at
    the end point; such a path may converge to a multiple isolated
    point of $X$ or a higher dimensional component of $X$. To track
    such paths, or even paths of higher multiplicity, one may employ
    special techniques which are described in detail in \cite{SW}.

In order to guarantee that all the isolated points of $X$ are among
the end points of solution paths of the homotopy, we must have some
control over the root count. This is the case for example if
$\pi^{-1}(1)$ is smooth and finite, $\pi^{-1}(t)$ is smooth and finite
for general $t \in \CC$ and in addition $|\pi^{-1}(1)|=|\pi^{-1}(t)|$
for general $t \in \CC$. This is the type of situation that we will
encounter in this paper. Note that even if $\pi^{-1}(t)$ is smooth and
finite for some $t \in \CC$, the corresponding subscheme of $\PP^m$
defined by homogenizing the system $H$ might have some nasty
components at infinity, possibly of higher dimension. This in effect
reduces the root count in $\CC^m$ in the light of B\'{e}zout's theorem
and Fulton's excess intersection formula \cite{F} and this can be
taken advantage of to set up more efficient homotopies with fewer
paths to follow than the B\'{e}zout number would suggest. More
generally we may consider only solutions contained in an open subset
$U \subseteq \PP^m$ and use this to reduce the number of homotopy
paths. We will do this in connection with over determined systems in
\secref{sec:squaring} and \secref{sec:main}.

In order for it to make sense to use a homotopy to solve the problem,
$\pi^{-1}(1)$ should be better understood or easier to handle somehow
than $\pi^{-1}(0)$. One may for example set up the homotopy $H$ in
such a way that $\pi^{-1}(1)$ has known isolated points. For the
homotopy studied in this paper the start point computation is not
trivial but of lower order complexity compared to the main problem of
approximating the isolated points of $\pi^{-1}(0)$, see
\thmref{thm:count} and \cororef{coro:bound1}.

Numerical homotopies often depend on parameters in such a way that
only a generic choice of these parameters yields a homotopy with
desired properties. In practice one uses random parameters which
results in probabilistic algorithms but since a general choice
suffices it is often said that the desired properties hold with
probability 1. Let $S^1 \subset \CC$ be the unit circle. The homotopy
presented in this paper depends on two parameters $p_0 \in \CC^n$ and
$\gamma \in S^1$ as well as the squaring of the systems for $X$ and
$Y$ explained in \secref{sec:squaring}. These parameters will be
assumed to be general, and in practice they are taken as random
vectors or matrices with complex entries.

\section{Euclidean distance degree and bottlenecks} \label{sec:EDD}
Let $X,Y \subseteq \CC^n$ be smooth subvarieties and consider the
closures $\bar{X},\bar{Y} \subseteq \PP^n$. The hyperplane at infinity
$H_{\infty}$ intersects $\bar{X}$ and $\bar{Y}$ in two subschemes
$X_{\infty}=\bar{X}\cap H_{\infty}$ and $Y_{\infty}=\bar{Y}\cap
H_{\infty}$. For a projective variety $Z \subseteq \PP^m$, we use
$\hat{Z} \subseteq \CC^{m+1}$ to denote the cone over $Z$.

The smooth quadric $Q_{\infty} \subset H_{\infty}$ defined by
$x_1^2+\dots +x_n^2=0$ is known as the isotropic quadric in
$\PP^{n-1}$. This choice of quadric induces a bilinear form on $\CC^n$
as well as an orthogonality relation: for $x,y \in \CC^n$ the
condition $x \perp y$ is equivalent to $x^Ty=0$ where $x$ and $y$ are
viewed as column vectors. This is the definition of orthogonality we
are using for our bottleneck problem. For a smooth variety $X
\subseteq \CC^n$ we will write $(T_xX)^{\perp}$ for the orthogonal
complement of the embedded tangent space $T_xX$ translated to the
origin. Similarly $z \perp T_xX$ for $z \in \CC^n$ means that $z$ is
orthogonal to $T_xX$ translated to the origin. The normal space at a
point $x \in X$ is by definition $(T_xX)^{\perp}$ translated to
$x$. For a general point $p_0 \in \CC^n$, $\{x \in X:(x-p_0) \perp
T_xX\}$ is finite and the number of points does not depend on
$p_0$. This number is called the Euclidean distance degree of $X$ and
is denoted $\EDD{X}$. For $X\times Y \subseteq \CC^{2n}$, we have that
$\EDD{X\times Y}=\EDD{X}\EDD{Y}$.

At a point $x \in X$ lines in the normal space through $x$ are
parameterized by $\PP^{n-d-1}$ where $d=\dim{X}$. There is an induced
map $\Phi:X \times \PP^{n-d-1} \rightarrow \Gr{2}{n+1}$ to the
Grassmannian of lines in $\PP^n$ which maps a line in $\CC^n$ to its
closure in $\PP^n$. Similarly there is a map $\Gamma:Y \times
\PP^{n-e-1} \rightarrow \Gr{2}{n+1}$ where $e=\dim{Y}$ and a product
map $$\Phi \times \Gamma: X \times Y \times \PP^{n-d-1} \times
\PP^{n-e-1} \rightarrow \Gr{2}{n+1} \times \Gr{2}{n+1}.$$ Let
$R=\pi((\Phi \times \Gamma)^{-1}(\Delta))$ where $\Delta \subset
\Gr{2}{n+1} \times \Gr{2}{n+1}$ is the diagonal and $\pi$ is the
projection to $X \times Y$. Note that there is an excess component of
$I(X,Y)$ in the form of $X\cap Y$ written as $(X\times Y) \cap
\Delta'$ where $\Delta' \subset \CC^n \times \CC^n$ is the diagonal
and as sets $I(X,Y)=(X\cap Y) \cup R$. We expect that
$\dim{\im{\Phi}}=\dim{\im{\Gamma}}=n-1$ and since
$\dim{\Gr{2}{n+1}}=2(n-1)$ we expect that $\dim{R}=0$.

\begin{ex} \label{ex:specialpos}
Let $Q \subset \PP^2$ be the isotropic quadric defined by
$x_0^2+x_1^2+x_2^2=0$ and let $p \in Q$. Let $a,b \in l$ be general
points on the tangent line $l$ to $Q$ at $p$. Now consider lines $X,Y
\subset \CC^3$ passing through $a$ respectively $b$ at infinity, but
otherwise general. We may assume that $X$ and $Y$ are not coplanar. In
this case $I(X,Y)=\emptyset$ and $\EDD{X}=\EDD{Y}=1$. To see that
$I(X,Y)$ is empty, suppose that $(x,y) \in I(X,Y)$ and let $v, w \in
\CC^3$ represent the directions of $X$ and $Y$. Note that $v,w$ are
homogeneous coordinates of the points $a,b \in \PP^2$. Then $(x-y)
\perp v,w$ and hence $(x-y)$ represents the line $l$ as well as the
point $p \in \PP^2$. This means that the closure $l' \subset \PP^3$ of
the line joining $x$ and $y$ intersects the hyperplane at infinity at
$p$ and the plane spanned by $l'$ and $l$ contains both $X$ and $Y$, a
contradiction.
\end{ex}

As we saw in the last example, the bottleneck problem may have a
kind of solutions at infinity which we are not interested in computing
in the present context.
\begin{defn}
Let $X,Y \subseteq \CC^n$ be smooth varieties. If there is a pair of
points $(x,y) \in \CC^n \times \CC^n$, with $x,y \neq 0$, $(x,y) \in
\hat{X}_{\infty} \times \hat{Y}_{\infty}$, $(x-y) \perp
T_x\hat{X}_{\infty}$ and $(x-y) \perp T_y\hat{Y}_{\infty}$ we will
call the pair a \emph{solution at infinity} to the bottleneck
problem.
\end{defn}

\begin{ex}
Let $A, B \subseteq \PP^{n-1}$ be smooth algebraic sets, possibly
reducible, and consider the cones $X=\hat{A} \setminus \{0\}$ and
$Y=\hat{B} \setminus \{0\}$. We can define the incidence $I(X,Y)$ and
the maps $\Phi$ and $\Gamma$ the same way as for varieties. In this
case, the residual set $R$ above is a cone in the sense that $(tx,ty)
\in R$ for all $(x,y) \in R$ and $t \in \CC^*$. However, if we
restrict $\Phi$ to $W=(X\cap H) \times \PP^{n-d-1}$ where $H \subset
\CC^n$ is a general hyperplane, we have that $\dim{\im{\Phi|_W}} \leq
n-2$ and we expect the images of $\Phi|_W$ and $\Gamma$ to be
disjoint. We thus expect the residual $R$ to be empty for cones. In
particular we can expect the bottleneck problem of two smooth
varieties $X, Y \subseteq \CC^n$ to have no solutions at infinity
beyond $\hat{X}_{\infty} \cap \hat{Y}_{\infty}$. See
\exref{ex:specialpos} for a pair of varieties in special position
where this is not the case.
\end{ex}

We now define the ED-correspondence $$\mathcal{E}_{X\times
  Y}=\{(x,y,p,q) \in \CC^{4n}: (x,y) \in X\times Y, \; (x-p)\perp
T_xX, \; (y-q)\perp T_yY\},$$ with its induced reduced structure. In
\cite{DHOST} the ideal defining $\mathcal{E}_{X\times Y}$ is given and
it is shown in Theorem 4.1 that it is an irreducible variety of
dimension $2n$. If there is a point $(x,y) \in X\times Y$ such that
$T_xX \cap \Ns{x}{X} = \{x\}$ and $T_yY \cap \Ns{y}{Y} = \{y\}$, the
projection $\mathcal{E}_{X\times Y} \rightarrow \CC^{2n}:(x,y,p,q)
\mapsto (p,q)$ is dominant and the generic fiber has $\EDD{X}\EDD{Y}$
points, see \cite{DHOST} Theorem 4.1. On the other hand, the
projection is dominant if and only if $\EDD{X}$ and $\EDD{Y}$ are both
non-zero.

\begin{ex}
Let $X \subset \CC^2$ be the line defined by $x-iy$ where
$i^2=-1$. Note that $X=T_xX=(T_xX)^{\perp}$ for every $x \in
X$. Therefore $\EDD{X}=0$. The same is true for any cone over a
subvariety of the isotropic quadric $\{x_1^2+\dots+x_n^2=0\}$ in
$\PP^{n-1}$.
\end{ex}

We will frequently make the following assumptions on smooth varieties
$X, Y \subseteq \CC^n$:
\begin{equation} \label{eq:assumptions}
\text{$\EDD{X}$, $\EDD{Y}$ are non-zero and $X_{\infty}$, $Y_{\infty}$
  are smooth and intersect $Q_{\infty}$ transversely.}
\end{equation}
Introduce a new variable $u$ and consider the closure
$N=\bar{\mathcal{E}}_{X\times Y} \subset \PP^{4n}$ and its
intersection $N_{\infty}$ with the hyperplane at infinity given by
$u=0$. For $p_0 \in \CC^n$, let $V_s(p_0)$ for $s \in \CC$ be the
family of $2n$-dimensional linear spaces in $\CC^{4n}$ defined by
$p=(1-s)y+sp_0$, $q=(1-s)x+sp_0$.
\begin{lemma} \label{lemma:EDDcorr}
  Let $X,Y$ be as in (\ref{eq:assumptions}) and let $N \subseteq
  \PP^{4n}$ be the closure of the ED-correspondence. Let $p_0 \in
  \CC^n$ be generic.
  \begin{enumerate}
  \item $\deg{N}=\EDD{X}\EDD{Y}$,
  \item the intersection $J_s=\mathcal{E}_{X\times Y} \cap V_s(p_0)$
    is transversal for generic $s \in \CC$ and for $s=1$,
  \item $|J_1|=|J_s|=\EDD{X}\EDD{Y}$ for generic $s \in \CC$,
  \item if the bottleneck problem of $X,Y$ has no solutions at
    infinity, then $N_{\infty} \cap \{x-q=y-p=0\} = \emptyset$.
  \end{enumerate}
\end{lemma}
\begin{proof}
Let $(x,y,p,q,0) \in N_{\infty}$. We first show the following:
\begin{equation} \label{eq:cond}
\text{If $x \neq 0$ then $x \in \hat{X}_{\infty}$ and $(x-p) \perp
  T_x\hat{X}_{\infty}$. If $y \neq 0$ then $y \in \hat{Y}_{\infty}$
  and $(y-q) \perp T_y\hat{Y}_{\infty}$.}
\end{equation}
Suppose that $x \neq 0$, the case $y \neq 0$ is similar. Let
$(x_i,y_i,p_i,q_i,u_i) \in \hat{N} \subseteq \CC^{4n+1}$ for $i \in
\NN$ be a sequence such that $x_i \rightarrow x$, $p_i \rightarrow p$
and $u_i \rightarrow 0$ but $u_i \neq 0$ for all $i$. Then $x_i/u_i
\in X$ and $(x_i/u_i-p_i/u_i) \perp T_{x_i/u_i} X$ for all $i$. Hence
$x \in \hat{X}_{\infty}$ and since $(x_i-p_i) \perp T_{x_i/u_i} X$ for
all $i$ it follows that $(x-p) \perp T_x\hat{X}_{\infty}$.

The intersection $J_1=\mathcal{E}_{X\times Y} \cap V_1(p_0)$ is
transversal, for example by generic smoothness and the fact that $p_0
\in \CC^n$ is generic. Clearly, $|J_1|=\EDD{X}\EDD{Y}$. Consider the
linear space $\bar{V}_1(p_0)=\{p-p_0u=q-p_0u=0\} \subset \PP^{4n}$. If
$(x,y,0,0,0) \in N_{\infty} \cap \bar{V}_1(p_0)$, then either $x \neq
0$ or $y \neq 0$. Say that $x\neq 0$. Then, by (\ref{eq:cond}), $x
\perp T_x\hat{X}_{\infty}$ which contradicts that $X_{\infty}$
intersects the isotropic quadric $Q_{\infty}$ transversely. The same
argument applies if $y \neq 0$. It follows that $N_{\infty} \cap
\bar{V}_1(p_0) = \emptyset$ and therefore
$\EDD{X}\EDD{Y}=|J_1|=\deg{N}$. Moreover, it follows that $V_s(p_0)$
intersects $\mathcal{E}_{X\times Y}$ transversely in $\EDD{X}\EDD{Y}$
points for generic $s \in \CC$ as well.

It remains to show the last statement. Let $(x,y,p,q,0) \in N_{\infty}
\cap \{x-q=y-p=0\}$. Then one of $x$ and $y$ is non-zero. If $y=0$
then $p=0$ and $x \perp T_x\hat{X}_{\infty}$ by (\ref{eq:cond}). This
contradicts that $X_{\infty}$ intersects the quadric $Q_{\infty}$
transversely. Hence $y \neq 0$ and similarly $x \neq 0$. Using
(\ref{eq:cond}) one more time we see that $x \in \hat{X}_{\infty}$, $y
\in \hat{Y}_{\infty}$, $(x-y) \perp T_x\hat{X}_{\infty}$ and $(y-x)
\perp T_y\hat{Y}_{\infty}$. Thus, $(x,y)$ represents a solution at
infinity of the bottleneck problem.
  \end{proof}

Let $V_0=\{x-q=y-p=0\} \subset \CC^{4n}$ and note that the projection
$\pi: \mathcal{E}_{X \times Y} \rightarrow X\times Y$ maps $V_0 \cap
\mathcal{E}_{X \times Y}$ bijectively onto $I(X,Y)$. This puts a
scheme-structure on $I(X,Y)$ which we will use. Suppose that $X$ and
$Y$ satisfy (\ref{eq:assumptions}). If $I(X,Y)$ has multiple points or
components of higher dimension, they contribute to the total
multiplicity $\EDD{X}\EDD{Y}$ and makes $|I(X,Y)_0|$ smaller than this
number. Since $(X\times Y) \cap \Delta' \subseteq I(X,Y)$ where
$\Delta' \subset \CC^n \times \CC^n$ is the diagonal, $X \cap Y$
contributes to the total multiplicity and if $X$ and $Y$ intersect in
higher dimension the number of isolated points of $I(X,Y)$ is smaller
than $\EDD{X}\EDD{Y}$. The following example of varieties not
satisfying (\ref{eq:assumptions}) shows that $|I(X,Y)_0|$ can be
bigger than $\EDD{X}\EDD{Y}$.

\begin{ex}
Let $X \subset \CC^3$ be the line $ix_1-x_2=x_3=0$. In this case
$\EDD{X}=0$ and the point $X_{\infty}$ lies on the quadric $Q_{\infty}
\subset \PP^2$. Note that $T_xX=X$ and $(T_xX)^{\perp}=N_xX=\langle
X,(0,0,1)\rangle$ for all $x \in X$. Now let $Y \subset \CC^3$ be a
general line. The line $Y$ intersects $\langle X,(0,0,1)\rangle$ in
exactly one point $y \in Y$ and the plane $N_yY$ intersects $\langle
X,(0,0,1)\rangle$ in a line which intersects $X$ in a point $x \in
X$. Hence $I(X,Y)$ consists of one point $(x,y)$ but
$\EDD{X}\EDD{Y}=0$.
\end{ex}

The following theorem states that the homotopy presented in this paper
is optimal if $X,Y$ satisfy (\ref{eq:assumptions}), $I(X,Y)$ is finite
and non-singular and the bottleneck problem has no solutions at
infinity.

\begin{thm} \label{thm:count}
If $X,Y \subseteq \CC^n$ are smooth varieties satisfying
(\ref{eq:assumptions}), $I(X,Y)$ is finite and non-singular and the
bottleneck problem has no solutions at infinity,
then $$|I(X,Y)|=\EDD{X}\EDD{Y}.$$
\end{thm}
\begin{proof}
Let $L=\{x-q=y-p=0\} \subset \PP^{4n}$. We have assumed that the
intersection $L \cap \mathcal{E}_{X\times Y}$ is transversal. By
\lemmaref{lemma:EDDcorr}, $L \cap N_{\infty} = \emptyset$ and
$|I(X,Y)|=\deg{N}=\EDD{X}\EDD{Y}$.
\end{proof}

\subsection{Reformulation in terms of classical invariants}
Let $X \subseteq \CC^n$ be a smooth subvariety satisfying the
assumptions (\ref{eq:assumptions}). Then $\EDD{X}$ may be expressed in
terms of the degrees $\mu_i(\bar{X})$ of the so-called polar classes
of $\bar{X}$. In fact, as is shown in \cite{DHOST} Theorem
6.11, $$\EDD{X}=\sum_{i=0}^d \mu_i(\bar{X}),$$ where
$d=\dim{X}$. Since $\bar{X}$ is smooth, this may be expressed in terms
of the degrees $c_i$ of the Chern classes of the tangent bundle of
$\bar{X}$: $$\EDD{X} = \sum_{i=0}^d(-1)^i \cdot (2^{d+1-i}-1)c_i.$$
See for example \cite{F} Example 14.4.15 for the definition of polar
classes and their relationship to Chern classes. As is explained in
\cite{P}, another way to phrase this is in terms of the degree of the
top Segre class of the Euclidean normal bundle:
$\EDD{X}=\deg{s_d(\N{X})}$.

\begin{ex} \label{ex:surfaces}
If $X \subseteq \CC^3$ is a general surface of degree $d$, then
$\EDD{X}=d^3-d^2+d$. Hence, if $X,Y \subseteq \CC^3$ are general
surfaces of degree $d$ and
$e$, $$|I(X,Y)_0|=(d^3-d^2+d)(e^3-e^2+e)-c,$$ where $c$ is the
contribution from $X\cap Y$. Based on experiments, it seems that
$c=de(d+e-1)$.
\end{ex}

\section{The homotopy} \label{sec:homotopy}
In this section we present the homotopy. As mentioned in the
introduction, the start points of the homotopy are built from
combining points in the normal loci of $X$ and $Y$ with respect to a
general point $p_0 \in \CC^n$. This will be explained in detail below
but for now let $p_0 \in \CC^n$ be a general point.

Suppose that $X \subset \CC^n$ is a smooth variety defined by an ideal
$(F_1,\dots,F_a) \subseteq \CC[x_1,\dots,x_n]$ and that $Y \subset
\CC^n$ is a smooth variety defined by $(G_1,\dots,G_b)$. For the sake
of presentation, assume for now that $X$ and $Y$ are complete
intersections and that $a=\codim{X}$ and $b=\codim{Y}$. The general
case will be addressed below. Note that the set of lines in a fiber
$\Ns{x}{X}$ that pass through the base point $x \in X$ may be viewed
as $\PP^{a-1}$. There is a map $X \times \PP^{a-1} \rightarrow
\PP^{n-1}$ which sends $(x,l)$, where $x \in X$ and $l$ is a line in
$\Ns{x}{X}$ through $x$, to the direction of $l$. Similarly there is a
corresponding map $Y \times \PP^{b-1} \rightarrow \PP^{n-1}$ for
$Y$. We will express these maps using the defining equations as
follows. Let $J_F$ and $J_G$ be the Jacobian matrices of
$F=(F_1,\dots,F_a)$ and $G=(G_1,\dots,G_b)$. Then, the normal line
maps described above are induced by the maps $f:X \times \CC^a
\rightarrow \CC^n:(x,v)\mapsto J_F(x)^Tv$ and $g:Y \times \CC^b
\rightarrow \CC^n:(y,w) \mapsto J_G(y)^Tw$. The homotopy is given by
the following equations:
\begin{equation} \label{eq:homocomplex}
\begin{array}{c}
  F(x)=0, \\
  G(y)=0, \\
  s(x-p_0)+(1-s)(x-y)-f(x,v)=0, \\
  s(y-p_0)+(1-s)(y-x)-g(y,w)=0,
\end{array}
\end{equation}
for $(x,y,v,w,s) \in \CC^n \times \CC^n \times \CC^a \times \CC^b
\times \CC$.

To define the path $\alpha:[0,1] \rightarrow \CC$ described in
\secref{sec:homomethods} we introduce a general parameter $\gamma \in
S^1 \subset \CC$ and use what in \cite{SW} is called the gamma
trick. Based on the gamma trick we may use the path $\alpha: [0,1]
\rightarrow \CC$, $\alpha(t) = \gamma t / (1+(\gamma-1)t)$, which for
a general $\gamma \in S^1$ ensures sufficient generality and avoids
degeneracy. Letting $s=\alpha(t)$ in (\ref{eq:homocomplex}) we get
\begin{equation} \label{eq:main}
\begin{array}{c}
  F(x)=0, \\
  G(y)=0, \\
  \gamma t(x-p_0)+(1-t)(x-y)-(\gamma t +1-t)\cdot f(x,v)=0, \\
  \gamma t(y-p_0)+(1-t)(y-x)-(\gamma t +1-t)\cdot g(y,w)=0,
\end{array}
\end{equation}
for $(x,y,v,w,t) \in \CC^n \times \CC^n \times \CC^a \times \CC^b
\times [0,1]$.

\begin{rem}
Note that (\ref{eq:main}) with $t=0$ is the system we get if we apply
Lagrange multipliers in the real setting to solve the optimization
problem of minimizing or maximizing $||x-y||^2$ under the constraints
$F(x)=G(y)=0$. We will compare our approach to more direct ways of
solving this system in \secref{sec:complexity}.
\end{rem}

\subsection{Equations for the start system} \label{sec:start}
A system of equations for the start points with respect to a point $p_0
\in \CC^n$ is given by $$F(x)=0, \; (x-p_0)=f(x,v),$$ in the case of $X$
and similarly for $Y$. These equations may be homogenized with respect
to the variables $v \in \CC^a$ to yield a system homogeneous and
linear in the these variables. The start points are $S_1 \times S_2$
where
\begin{equation} \label{eq:start}
  \begin{array}{l}
  S_1=\{(x,v) \in X \times \CC^a: (x-p_0)=f(x,v)\}\\
  S_2=\{(y,w) \in Y \times \CC^b: (y-p_0)=g(y,w)\}.
  \end{array}
\end{equation}
Since we have ordered the coordinates of the homotopy as $(x,y,v,w)$,
this has to be read as $S_1 \times S_2 = \{(x,y,v,w):(x,v) \in S_1, \;
(y,w) \in S_2\}$.

\subsection{Effects of squaring the system} \label{sec:squaring}
If $X$ is defined by more than $a=\codim{X}$ equations we need to
square the system which we do by replacing defining equations
$(\hat{F}_1,\dots,\hat{F}_r)$ for $X$ by $a$ general linear
combinations $F=(F_1,\dots,F_a)$ of $(\hat{F}_1,\dots,\hat{F}_r)$.
Similarly, given defining equations $(\hat{G}_1,\dots,\hat{G}_s)$ for
$Y$ such that $s > b=\codim{Y}$, these are replaced by $b$ general
linear combinations $G=(G_1,\dots,G_b)$ of
$(\hat{G}_1,\dots,\hat{G}_s)$. The homotopy is then given by
(\ref{eq:main}) applied to the systems $F$ and $G$. In the case $X=Y$
defined by $(\hat{F}_1,\dots,\hat{F}_r)$ we may take $F=G$. This is
convenient as there is no need to solve the start system twice in this
case.

The squaring of the system has the effect that $F$ and $G$ define some
subschemes $X',Y' \subseteq \CC^n$ with $X \subseteq X'$ and $Y
\subseteq Y'$. Suppose that $X, Y \neq \emptyset$. By Bertini's
theorem, $X' \times Y'$ is equidimensional of dimension
$\dim{X}+\dim{Y}$ and $X\times Y$ is an irreducible component of $X'
\times Y'$. Let $\Sigma_1=\{x \in X: p_0 \in \Ns{x}{X}\}$ and
$\Sigma_2=\{y \in Y: p_0 \in \Ns{y}{Y}\}$. We may assume that the
finite subsets $\Sigma_1 \times \Sigma_2$ and $I(X,Y)_0$ of $X\times
Y$ are inside the smooth locus of $X' \times Y'$. It follows that
$(\Sigma_1 \times \Sigma_2) \cap W = I(X,Y)_0 \cap W = \emptyset$ for
any irreducible component $W \subseteq X' \times Y'$ other than
$X\times Y$. In particular, this means that every isolated point of
$I(X,Y)$ lifts to an isolated solution of the squared system
(\ref{eq:homocomplex}) over $s=0$.

Considering the equations for the start system in the general case of
over determined systems, we can still use the equations given by
$F(x), (x-p_0)-f(x,v)$ and $G(y), (y-p_0)-g(y,w)$
described in \secref{sec:start}. However, we must pick only points
$(x,y,v,w) \in X \times Y \times \CC^a \times \CC^b$ as in
(\ref{eq:start}). Since the equations defining $X$ and $Y$ are part of
the input, this can be done by a simple filtering process by
evaluating the defining equations on tentative start points. One may
also solve a corresponding over determined system with standard
homotopy methods. Either way, the since start point computation is of
lower order complexity it is not the focus of this paper.

\subsection{Isolated points of the bottleneck problem} \label{sec:main}
Let $S_1$ and $S_2$ be as in (\ref{eq:start}). Then $|S_1| = \EDD{X}$
and $|S_2|=\EDD{Y}$. Hence the number of start points $S_1 \times S_2$
is $\EDD{X}\EDD{Y}$.

\begin{thm} \label{thm:paths}
Let $X,Y \subset \CC^n$ be smooth subvarieties that satisfy
(\ref{eq:assumptions}). The homotopy (\ref{eq:main}) has
$\EDD{X}\EDD{Y}$ solution paths starting at $S_1 \times S_2$ whose
endpoints after projection to $\CC^{2n}$ include all the isolated
points of $I(X,Y)$. The paths are non-singular except possibly at the
end points.
\end{thm}
\begin{proof}
Using the notation of \secref{sec:squaring}, $X\times Y$ is a
component of the subscheme $X'\times Y' \subset \CC^n \times \CC^n$
defined by the system $(F,G)$. Let $U_0 = \CC^n \times \CC^n \setminus
\overline{X'\times Y' \setminus X\times Y}$ and $U=U_0 \times
\CC^{a+b}$. By \secref{sec:squaring}, we may assume that the isolated
points of $I(X,Y)$ are contained in $U_0$ and that $S_1\times S_2
\subset U$. For $s=1$, the set of solutions to (\ref{eq:homocomplex})
contained in $U$ is equal to $S_1\times S_2$ and those solutions are
smooth. Since $|S_1\times S_2|=\EDD{X}\EDD{Y}$, we have to show that
for general $s \in \CC$, the number of solutions to
(\ref{eq:homocomplex}) contained in $U$ is equal to
$\EDD{X}\EDD{Y}$. Once this is established, the statement follows from
the general theory of parameter homotopies, see for example \cite{SW}
Theorem 7.1.6 together with the gamma trick \cite{SW} Lemma 7.1.3.

Using the notation of \secref{sec:EDD}, let $\mathcal{E}_{X\times Y}$
be the ED-correspondence, $V_s(p_0)$ the family of linear spaces
defined there and let $\pi: \mathcal{E}_{X\times Y} \rightarrow
\CC^n\times \CC^n$ be the projection $(x,y,p,q) \mapsto (x,y)$. For
generic $s \in \CC$, there is a one-to-one correspondence between the
solutions to (\ref{eq:homocomplex}) contained in $U$ and
$\mathcal{E}_{X\times Y} \cap V_s(p_0) \cap \pi^{-1}(U_0)$. It is
given by the map $\CC^n \times \CC^n \times \CC^a \times \CC^b
\rightarrow \CC^{4n}$ defined by
  \[
    (x,y,v,w)
    \mapsto
  \begin{pmatrix}
    x\\y\\(1-s)y+sp_0\\(1-s)x+sp_0
    \end{pmatrix}.
  \]
Since $\mathcal{E}_{X\times Y} \cap V_1(p_0) \subseteq \pi^{-1}(U_0)$
we have that $\mathcal{E}_{X\times Y} \cap V_s(p_0) \subseteq
\pi^{-1}(U_0)$ for generic $s \in \CC$ as well. By
\lemmaref{lemma:EDDcorr}, $|\mathcal{E}_{X\times Y} \cap V_s(p_0)|$ is
finite and equal to $\EDD{X}\EDD{Y}$ for generic $s \in \CC$.
\end{proof}

\begin{coro} \label{coro:bound1}
Let $X,Y \subset \CC^n$ be smooth varieties that satisfy
(\ref{eq:assumptions}). Then $|I(X,Y)_0| \leq \EDD{X}\EDD{Y}$.
\end{coro}

\begin{rem}
If $X=Y$ with $\dim{X}>0$ it is enough to form the start points from
all pairs $(s_1,s_2)$ with $s_1,s_2 \in S_1=S_2$, $s_1\neq s_2$ and
without taking the order into account. We can still generate all the
isolated points of $I(X,X)$ from the end points of the solution paths
using the action $(x,y) \mapsto (y,x)$ on $I(X,X)$. This is because
the homotopy is in this case invariant under the symmetry $(x,y,v,w)
\mapsto (y,x,w,v)$. Of course, $(x,y) \in I(X,X)$ and $(y,x) \in
I(X,X)$ with $x\neq y$ represent the same bottleneck, which is what
we are really interested in. Any start point that is fixed under the
symmetry, that is it comes from a pair $(s_1,s_1)$ with $s_1=(x,v) \in
S_1$, can be discarded as the corresponding solution path projected to
$X \times X$ is constant with value $(x,x)$. But $(x,x) \in I(X,X)$ is
not isolated, unless $X$ is a point.
\end{rem}

\begin{coro} \label{coro:bound2}
Let $X \subset \CC^n$ be a smooth variety that satisfies
(\ref{eq:assumptions}). Then the number of isolated bottlenecks to
$X$ is bounded by ${\EDD{X} \choose 2}$.
\end{coro}

\section{Linear projections} \label{sec:projections}

In applications it is often relevant to consider smooth maps
$\pi:\CC^n \rightarrow \CC^m$ where $m<n$ such that $\pi_{|X}$ and
$\pi_{|Y}$ are embeddings of $X$ and $Y$ with smaller codimension. For
example, algorithms in the spirit of marching cubes \cite{CL} are
sensitive to the dimension of the ambient space $n$ and are typically
exponential in $n$. For these reasons it is useful to have a
formulation of the homotopy that computes the isolated points of
$I(\pi(X),\pi(Y))$ rather than $I(X,Y)$. The assumption is that we
have equations for $X$ and $Y$ and that the map $\pi$ is given
although equations for $\pi(X)$ and $\pi(Y)$ might not be
available. On $\CC^n$ and $\CC^m$ we have bilinear forms $\langle
\cdot, \cdot \rangle$ induced by the standard scalar products on
$\RR^n$ and $\RR^m$. For $x \in X$ consider the differential
$d\pi_x:T_x\CC^n \rightarrow T_{\pi(x)}\CC^m$ which induces as
isomorphism between the embedded tangent spaces $T_xX \rightarrow
T_{\pi(x)}\pi(X)$, seen as subspaces of $T_x\CC^n \cong \CC^n$ and
$T_{\pi(x)}\CC^m \cong \CC^m$. For $x \in X$, $p_0 \in \CC^n$ and $z
\in T_xX$ we have that $\langle \pi(x)-\pi(p_0), d\pi_x(z)
\rangle=\langle \adj{d\pi_x}(\pi(x)-\pi(p_0)),z \rangle$ where
$\adj{d\pi_x}$ is the adjoint map. To find the normal locus of
$\pi(X)$ with respect to $\pi(p_0)$ we may thus compute the set of
points $x \in X$ such that $\adj{d\pi_x}(\pi(x)-\pi(p_0)) \perp
T_xX$. In a similar fashion we get the equations of the main homotopy
in this setting:
\[
\begin{array}{c}
  F(x) =0,\\ G(y) =0,\\ \adj{d\pi_x}[\gamma
    t(\pi(x)-\pi(p_0))+(1-t)(\pi(x)-\pi(y))]-(\gamma t +1-t) \cdot
  f(x,v)=0, \\ \adj{d\pi_y}[\gamma
    t(\pi(y)-\pi(p_0))+(1-t)(\pi(y)-\pi(x))]-(\gamma t +1-t) \cdot
  g(y,w)=0.
\end{array}
\]

Now assume that $\pi:\CC^n \rightarrow \CC^m$ is a linear map whose
restriction to $X$ and $Y$ are embeddings. In the case of a smooth
variety $X=Y$, these assumptions are met by a general linear map
$\CC^n \rightarrow \CC^m$ where $m=2\cdot \dim{X}+1$ and a simple way
to reduce ambient dimension in practice in case $n>2 \cdot \dim{X}+1$
is to use a random linear map. In \exref{ex:curves} we will consider
real curves in higher codimension and project these to $\RR^3$ for
visualization purposes. In the case where $\pi$ is linear and
represented by an $m\times n$-matrix $M$, the homotopy above becomes:

\[
\begin{array}{c}
  F(x)=0, \\
  G(y)=0, \\
  M^TM(\gamma t(x-p_0)+(1-t)(x-y))-(\gamma t +1-t) \cdot f(x,v)=0, \\
  M^TM(\gamma t(y-p_0)+(1-t)(y-x))-(\gamma t +1-t) \cdot g(y,w)=0.
\end{array}
\]

\section{Examples} \label{sec:examples}

For the implementation of the examples below we have used
\emph{Bertini} \cite{bertini} and \emph{Macaulay2} \cite{M2} with the
package \cite{M2bertini}. The timings were done using a 2.8 GHz
processor of type Intel i7-2640M. More elaborate timings are performed
in \secref{sec:complexity}. The plotting was done using
\emph{Matplotlib} \cite{matplotlib}.

\begin{ex} \label{ex:curves}
Consider a complete intersection curve $X \subseteq \RR^n$ defined by
$n-2$ random polynomials of degree 2 in $\RR[x_1,\dots,x_n]$ and the a
hyperellipsoid $\sum_{i=1}^n r_ix_i^2-1$ where $0 \leq r_1 \leq 1$ are
random. The reason for including the hyperellipsoid is to ensure
compactness of $X$. In addition to this we require that the random
quadrics pass through a random point on the hyperellipsoid, this is to
ensure that $X$ is not empty. The method to compute bottlenecks
described above was performed with both input varieties equal to
$X_{\CC}$, with $X$ as above for $n=2$, $n=3$, $n=4$ and
$n=7$. \figref{fig:curves} displays plots of some of the results. For
visualization purposes we have performed the method subject to a
random orthogonal projection to $\RR^3$ in the cases $n>3$ as in
\secref{sec:projections}. In the figure, the computed bottlenecks
are represented by a pair of red points on $X$ together with the line
segment that joins them. We have also sampled $X$ (blue points) to
visualize the curve. Plots without the normal lines are also included.

To get a sense of the timing, in the case of $n=4$ the start point
homotopy follows 32 paths and the main homotopy follows 496 paths. The
computation takes 2-3 seconds in this case. In the case $n=7$ the
start homotopy follows 448 paths and the main homotopy follows 100128
paths. The number of real bottlenecks can however be quite small,
in the random example tested it was only 6. The computation takes
20-25 minutes in this case.
\begin{figure}[ht]
\centering
\caption{Plots of complete intersection curves.}
\label{fig:curves} 
\subfloat[][A conic in $\RR^2$.]{
  \includegraphics[trim={2cm 3cm 2cm 3cm},clip,scale=0.3]
                  {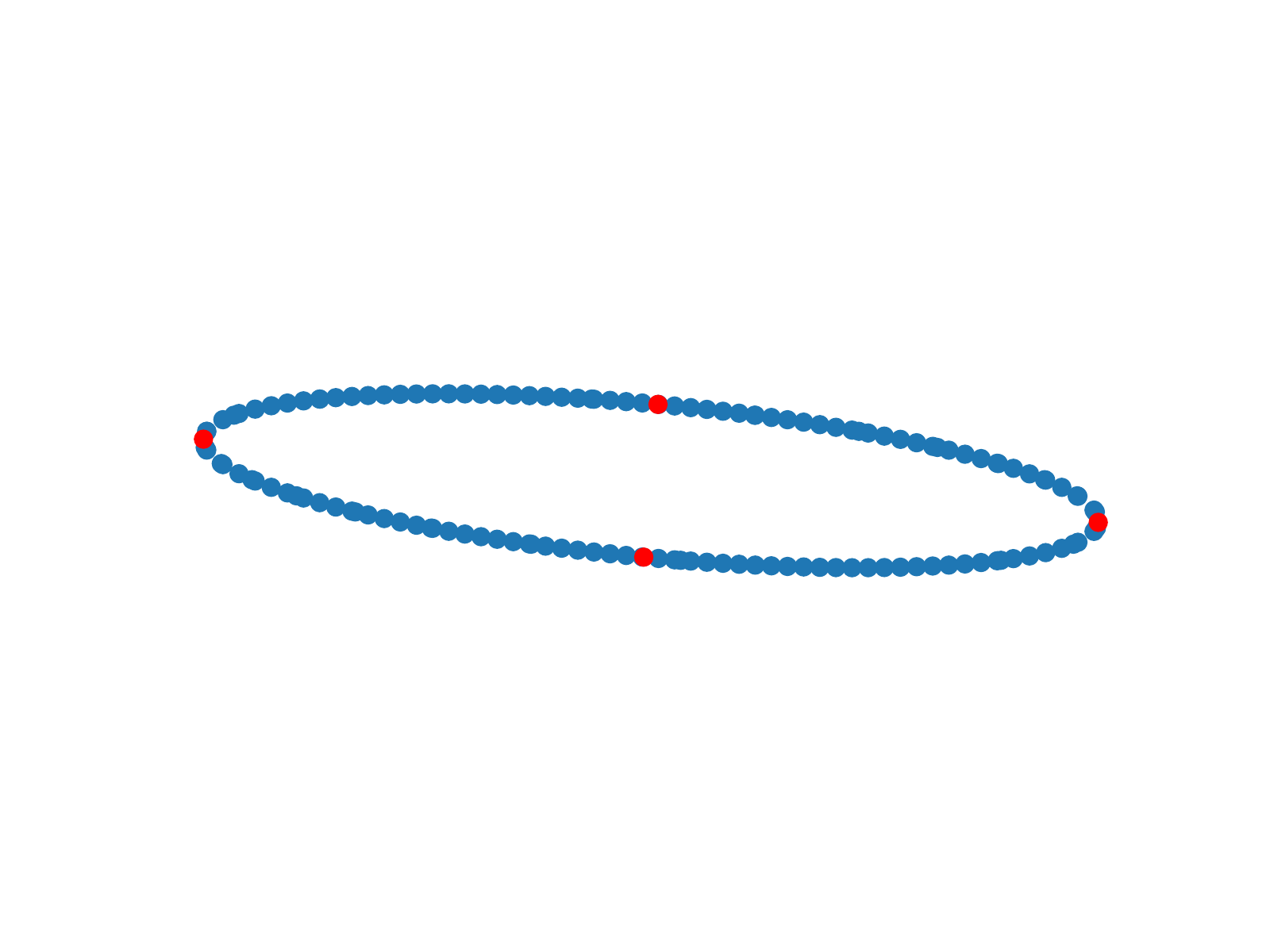}}
\qquad
\subfloat[][Bottlenecks.]{
  \includegraphics[trim={2cm 3cm 2cm 3cm},clip,scale=0.3]
                  {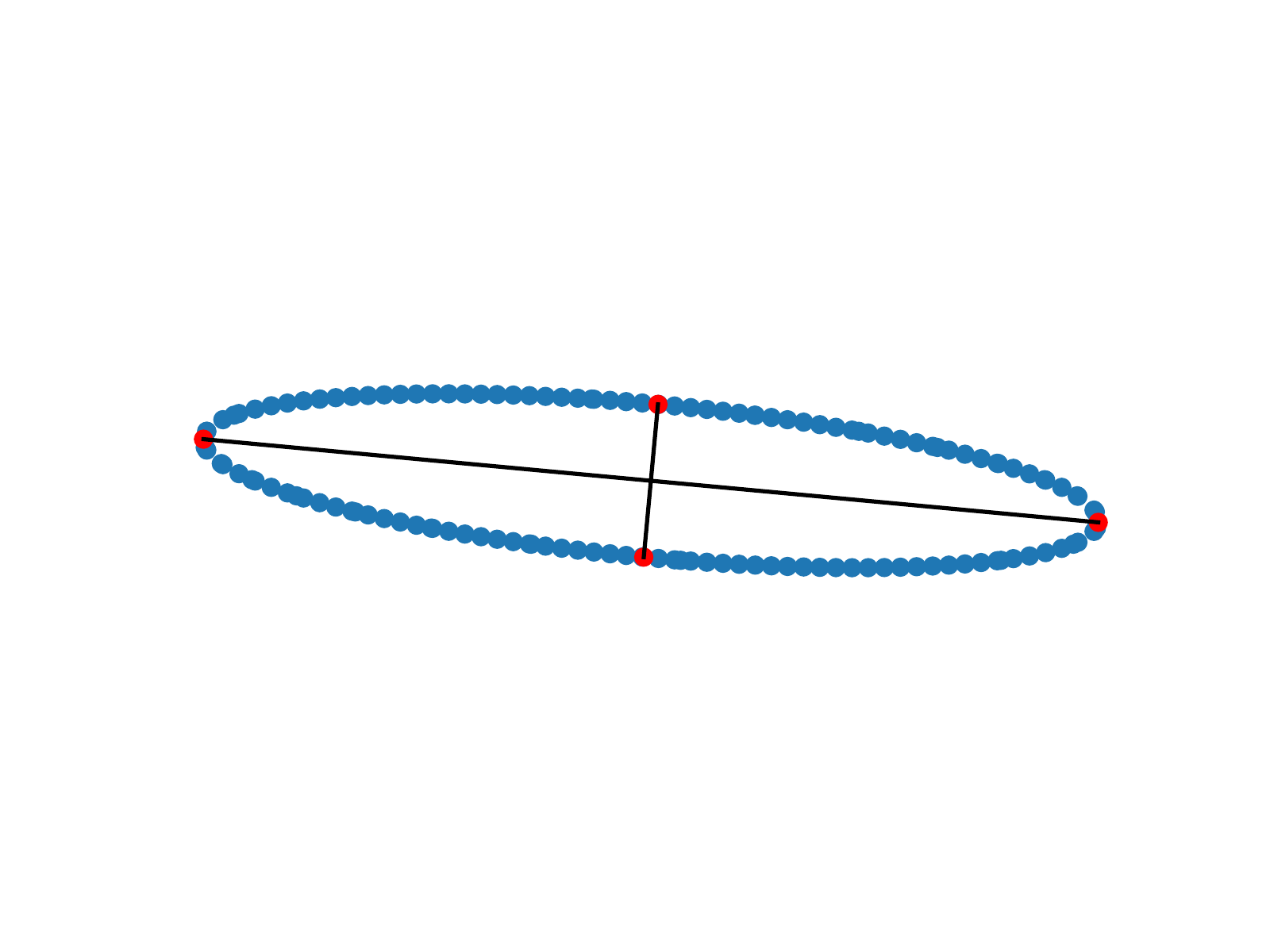}}
\qquad
\subfloat[][A curve in $\RR^3$.]{
  \includegraphics[trim={2cm 3cm 2cm 2cm},clip,scale=0.38]
                  {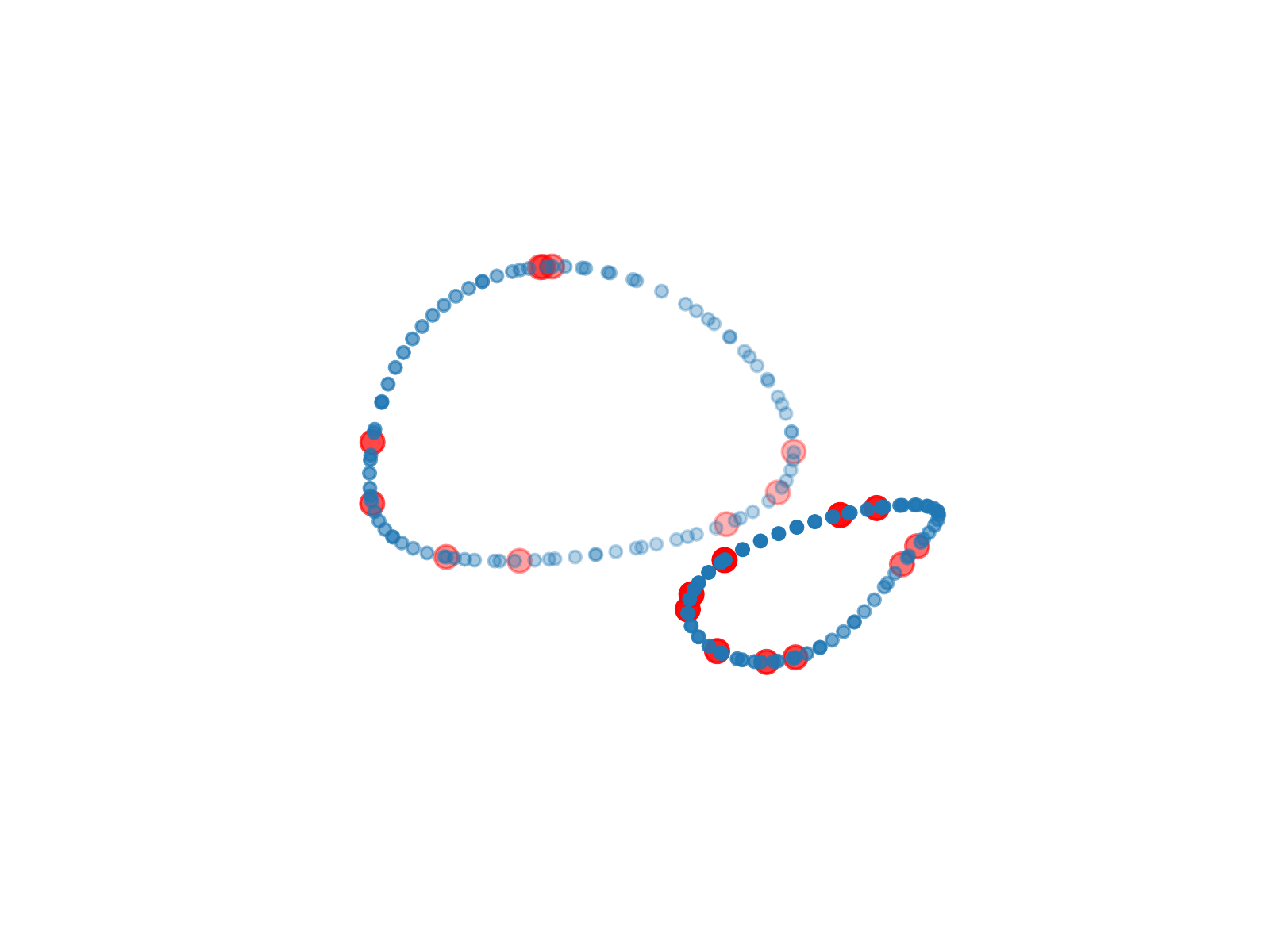}}
\qquad
\subfloat[][Bottlenecks.]{
  \includegraphics[trim={2cm 3cm 2cm 2cm},clip,scale=0.38]
                  {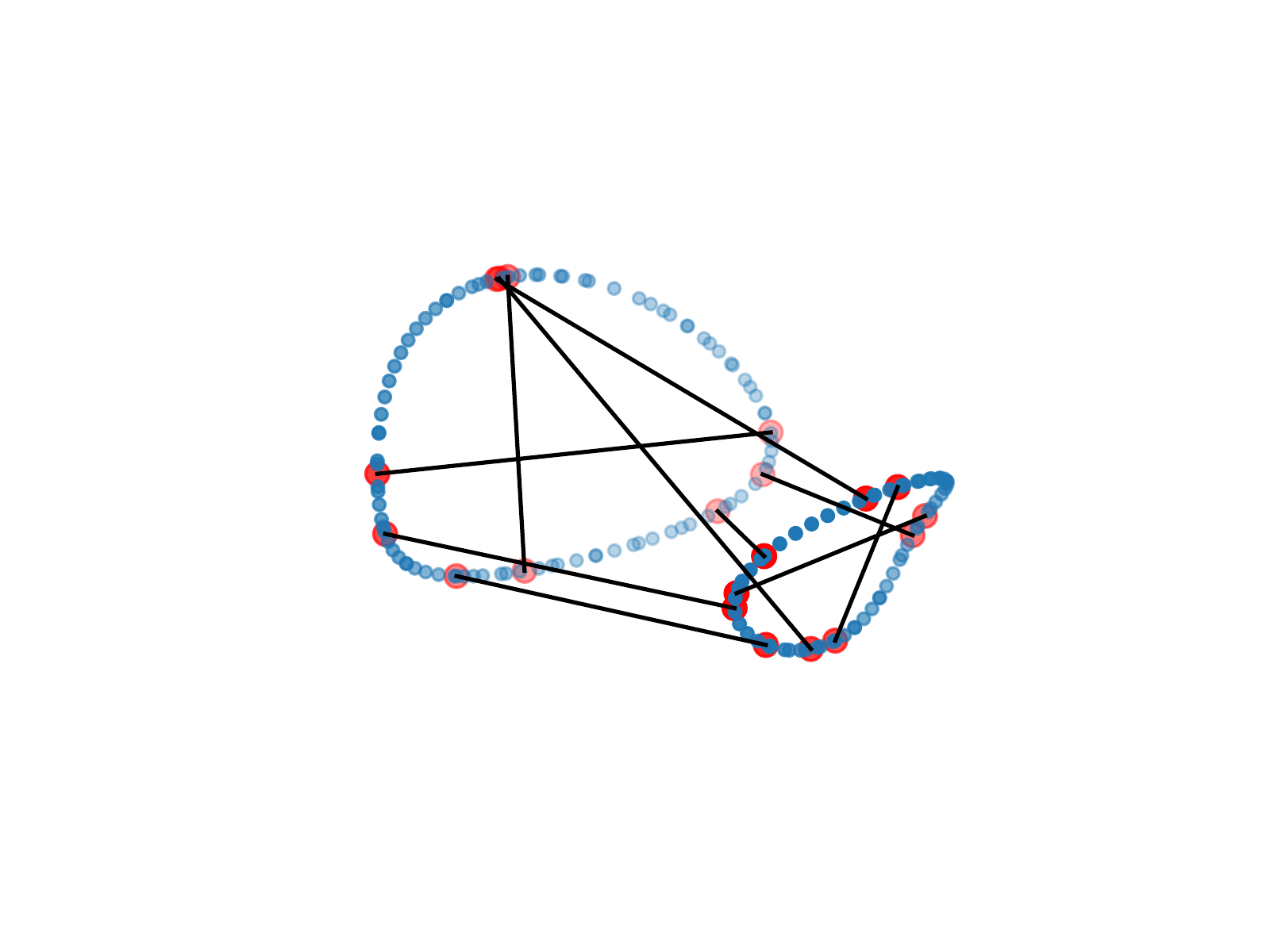}}
\qquad
\subfloat[][A curve projected from $\RR^4$.]{
  \includegraphics[trim={2cm 3cm 2cm 2cm},clip,scale=0.35]
                  {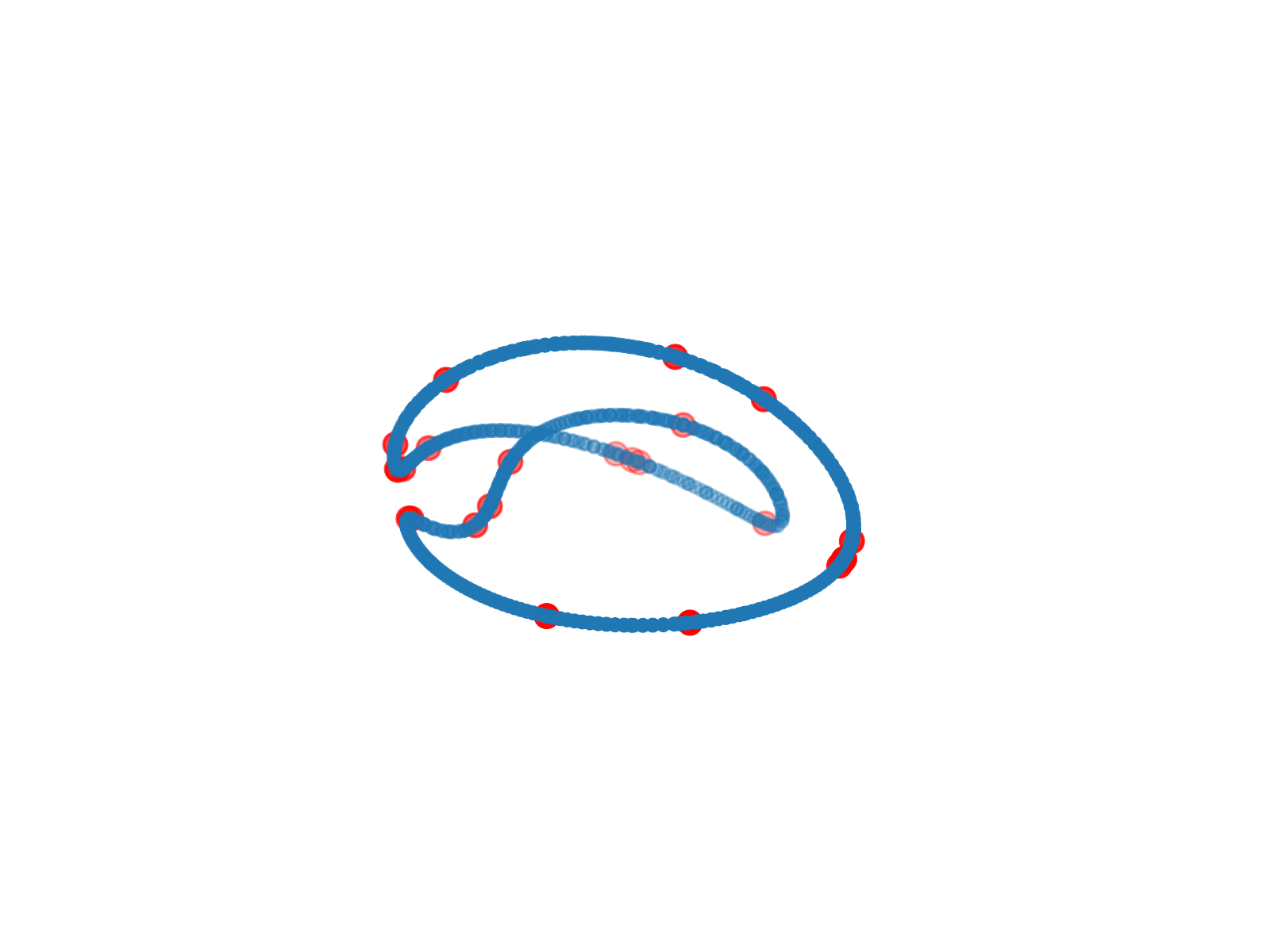}}
\qquad
\subfloat[][Bottlenecks.]{
  \includegraphics[trim={2cm 3cm 2cm 2cm},clip,scale=0.35]
                  {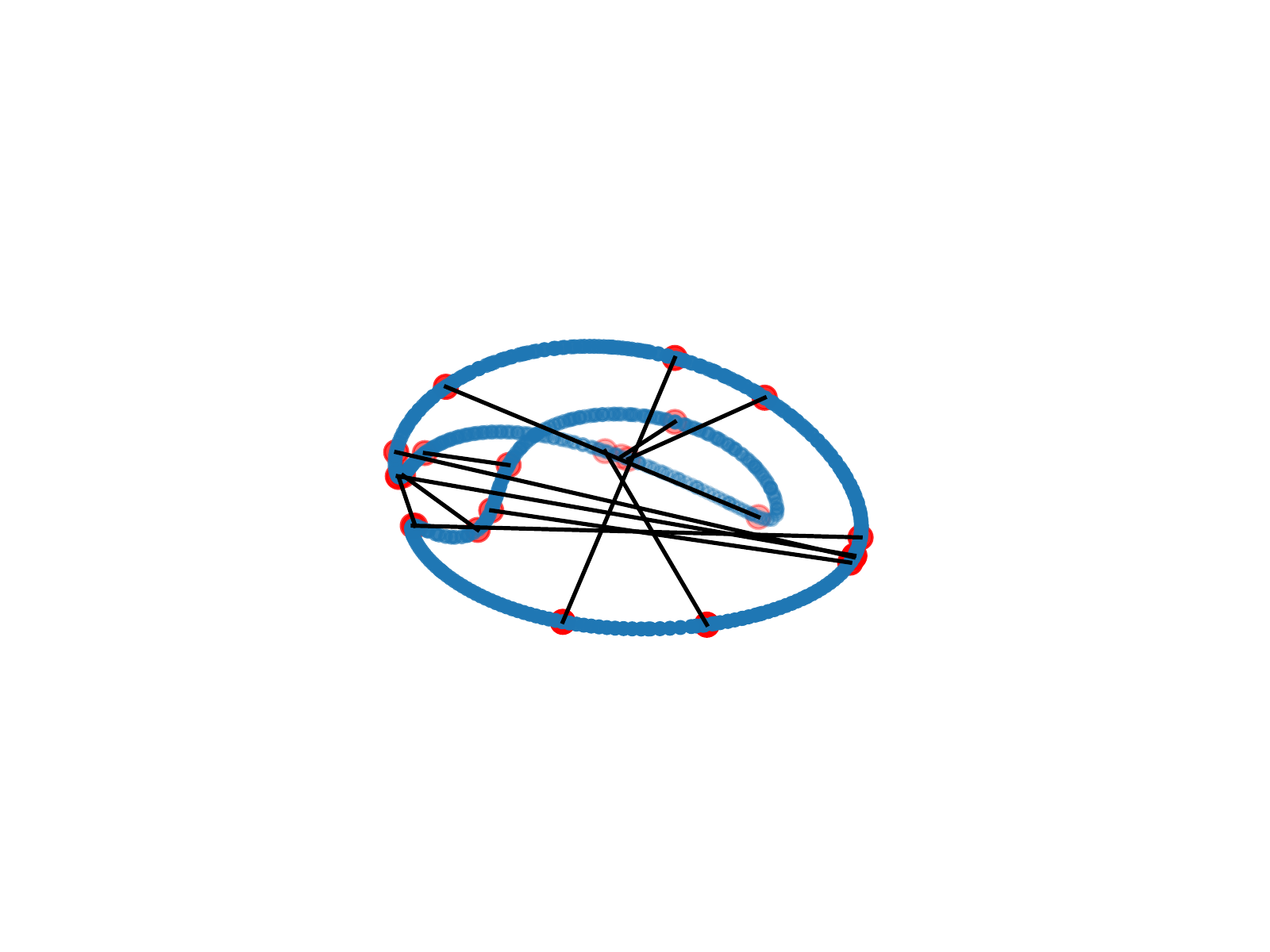}}
\end{figure}
\end{ex}

\exref{ex:curves} in a sense exemplifies the worst case scenario as
the start point homotopy for these random complete intersection curves
is optimal and has as many paths as there are solutions to the start
system (this follows from the discussion in \secref{sec:complexity},
see also \cite{DHOST} Corollary 2.10). More benefit from the method is
obtained in cases where structure is detected during the start system
run, embodied in the fact that the start system has fewer solution
than expected. This in turn means that fewer paths have to be followed
for the main homotopy. This is illustrated in \secref{sec:complexity}.

\begin{ex}
The method to compute bottlenecks was performed on the
complexification of the Goursat surface in $\RR^3$ defined
by $$x^4+y^4+z^4+(x^2+y^2+z^2)^2-2(x^2+y^2+z^2)-3=0.$$ The start point
homotopy follows 108 paths and there are 52 solutions. This means that
1326 paths are followed by the main homotopy. There are 13 real
bottlenecks which are plotted together with a sampling of the
surface in \figref{fig:goursat}. In total the computation takes about
6 seconds.
\begin{figure}[ht]
     \centering
     \caption{A surface in $\RR^3$.}
     \label{fig:goursat}
     \subfloat[][A quartic surface in $\RR^3$.]{
       \includegraphics[scale=0.3]{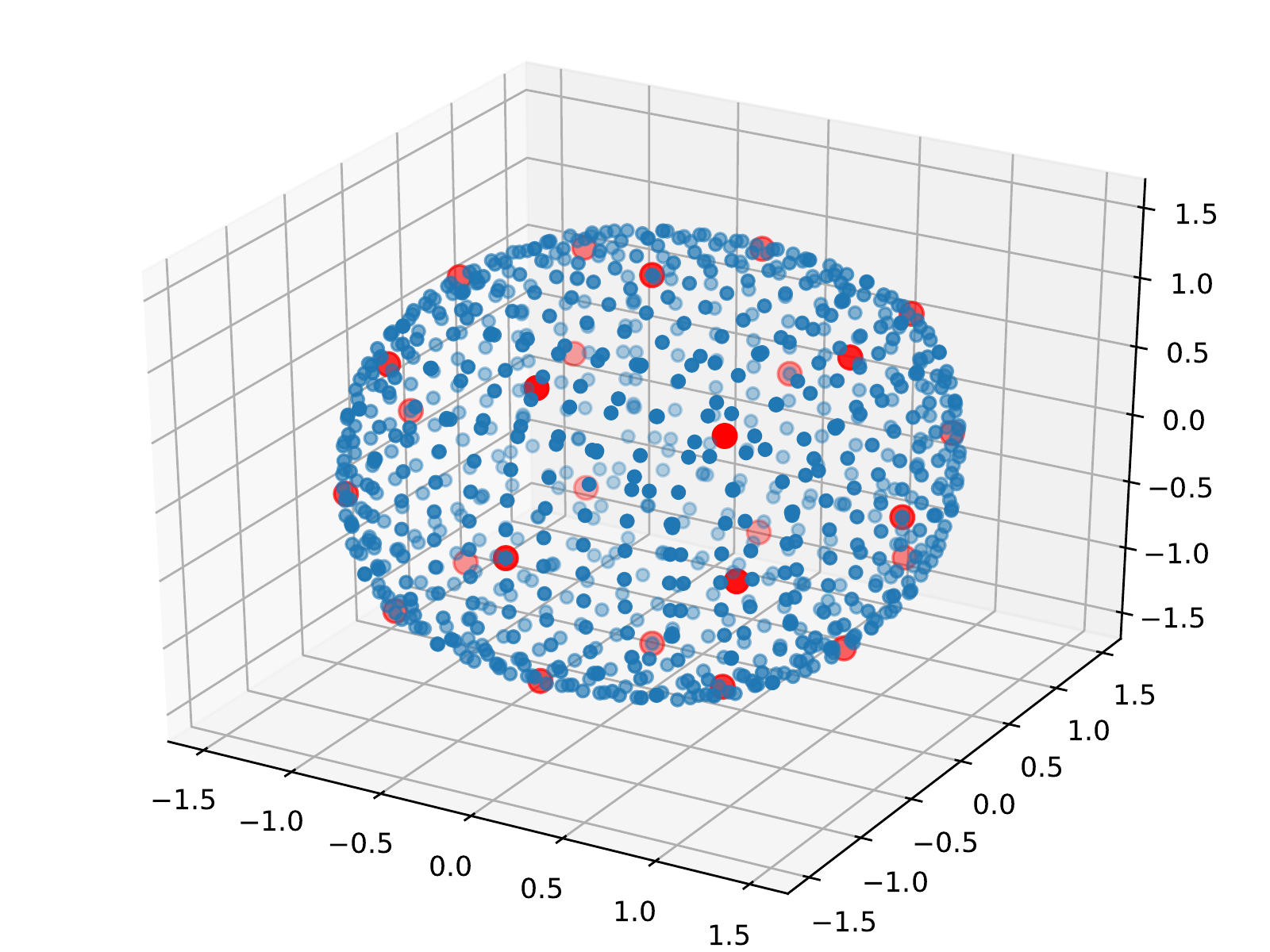}}
     \qquad
     \subfloat[][Bottlenecks.]{
     \includegraphics[scale=0.3]{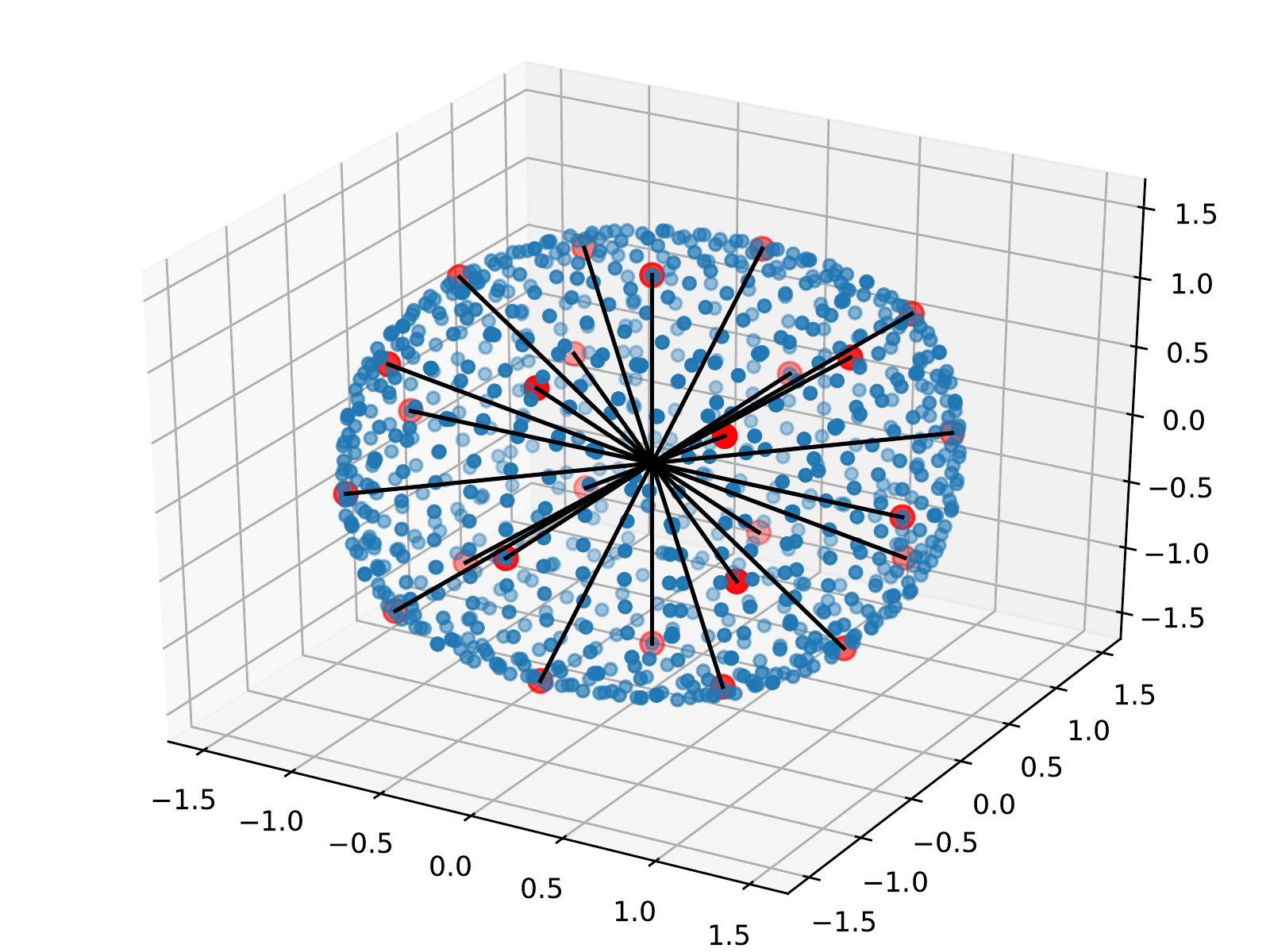}}
\end{figure}
\end{ex}

\begin{ex}
The following example originates from computational chemistry and
explores the geometry of molecular conformation spaces. A cycloheptane
molecule is a ring of seven coal atoms with a pair of hydrogen atoms
bonded to each coal atom. Cycloheptane is a special case of a
cycloalkane molecule and we will use a mechanical model of such
molecules from the computational chemistry literature
\cite{cyclo-8}. For the geometry of this model only the positions of
the coal atoms in space are relevant. One way of viewing the model is
to consider a conformation, or configuration, of the molecule as an
embedding into $\RR^3$ of the cycle-graph with 7 vertices, subject to
constraints. The constraints are given by fixing the edge lengths,
that is the distance between two consecutive atoms in the ring, as
well as the angle between consecutive edges, that is the bond
angles. In this model all the edge lengths are equal and all the bond
angles are equal. We may set the edge lengths to 1 and the bond angle
for this example was chosen to be 115 degrees, see \cite{CH}
Chapter~25.2 for a discussion of bond angles of cycloalkane molecules.

The variety $X$ parameterizing conformations may be considered in
$\RR^{21}$ with three coordinates for each of the seven vertices of
the graph. All in all we have seven distance constraints and seven
angle constraints. In addition we are only interested in
configurations up to rigid motion, that is up to the action of the
six-dimensional special Euclidean group. We thus expect the variety of
all conformations to be a curve, which it is. One way to deal with
rigid motions is to fix three of the vertices in a way that is
consistent with the distance and angle constraints. This way we get
rid of nine variables and three of the constraints will be
automatically satisfied. This gives us a curve in $\RR^{12}$ defined
by eleven polynomials. Moreover, two of the remaining angle
constraints are now linear and it is therefore natural to consider the
corresponding curve $X \subset \RR^{10}$ defined nine quadratic
polynomials.

One of the most basic questions about the geometry of $X$ is how many
connected components it has. In this example, we apply the method for
computing the isolated points of $I(X_{\CC},X_{\CC})$ to address this
question.

In this context one may ask if the complexification $X_{\CC}$ of $X$
satisfies the assumptions made throughout this paper. It is easy to
verify with techniques from numerical algebraic geometry
\cite{bertini, BHSW, SW} that $X_{\CC}$ is an irreducible curve of
degree 112. Smoothness of $X_{\CC}$ can be verified numerically as
well, at least if done with some care. Instead of considering an ideal
generated by minors of the Jacobian matrix $J$ it is better to
introduce auxiliary variables $v=(v_1,\dots,v_{10})$ and a random
linear form $l(v) \in \CC[v_1,\dots,v_{10}]$ and express the rank
condition on $J$ as nine bilinear equations $Jv^T=0$ together with
$l(v)=0$. Alternatively, one may introduce variables
$v=(v_1,\dots,v_9)$ and express the rank condition as
$vJ=0$. Regarding the assumptions (\ref{eq:assumptions}), the
assumption that the Euclidean distance degree is positive is tested
during the procedure. To avoid intersection with the isotropic quadric
at infinity we make a random change of coordinates, that is we change
coordinates using a real $10\times 10$-matrix. It is enough to use a
random real diagonal matrix, which is preferable in order to preserve
sparseness. One should be aware that changing coordinates by a
non-orthogonal linear transformation does not in general respect
bottlenecks but it does preserve the number of components of
$X$. The assumption that $(X_{\CC})_{\infty}$ is smooth in
\lemmaref{lemma:EDDcorr} is convenient but a bit stronger than
necessary, it is enough that $X_{\CC}$ is disjoint from the isotropic
quadric at infinity. Finally, one may easily check numerically that
$I(X_{\CC},X_{\CC})$ has no higher dimensional component except for
$X_{\CC}$. If that were the case, the normal space at a general point
$x \in X_{\CC}$ would intersect $X_{\CC}$ in some point $y$ such that
the line joining $x$ and $y$ is normal to $X_{\CC}$ at $y$.

The homotopy (\ref{eq:main}) was run with both input varieties equal
to $X_{\CC}$ after changing coordinates as above. This gives us a
lower bound for the distance between two connected components of $X$,
namely $b = \min\{||x-y||: (x,y) \in
(I(X_{\CC},X_{\CC})_0)_{\RR}\}$. To get a sense of the timing, the
start point homotopy follows 5120 paths and there are 448
solutions. This means that the main homotopy follows 100128 paths. The
whole computation takes about 2 h.

Now, given the lower bound $b$ on the distance between connected
components it is straightforward to compute the number of components
of $X$ with a sampling procedure and building the Vietoris-Rips
complex with parameter $r < b/2$. Given a finite sample $E \subset X$
of $X$ this is simply the graph with vertex set $E$ and edges
$\{(e_1,e_2) \in E\times E:0<||e_1-e_2||< 2r\}$. The sampling can be
done with standard homotopy methods. We will not go into details of
this procedure in the present paper but the density of the sample has
to be high enough, as measured with respect to the distance in the
ambient space. Namely, we need to guarantee that given any point of
$X$ there is a point of $E$ at distance less than $r$. It is
straightforward to set up a sampling procedure that guarantees this by
intersecting $X$ with a fine enough grid of hyperplanes in
$\RR^{10}$. This procedure was carried out and the result is that the
curve $X$ has two connected components. Of course, several steps in
this computation are subject to numerical errors, a subject we will
not go into here (see \cite{DEHH} for a related
discussion). \figref{fig:cycloheptane} shows the curve before the
change of coordinates. The curve is drawn using so-called torsion
angles (see \cite{cyclo-8}) as coordinates and a random orthogonal
projection of these to $\RR^3$.
\begin{figure}[ht]
    \centering
    \caption{Cycloheptane conformation curve
      in torsion angles projected to $\RR^3$.}
    \label{fig:cycloheptane} 
      \includegraphics[scale=0.5]{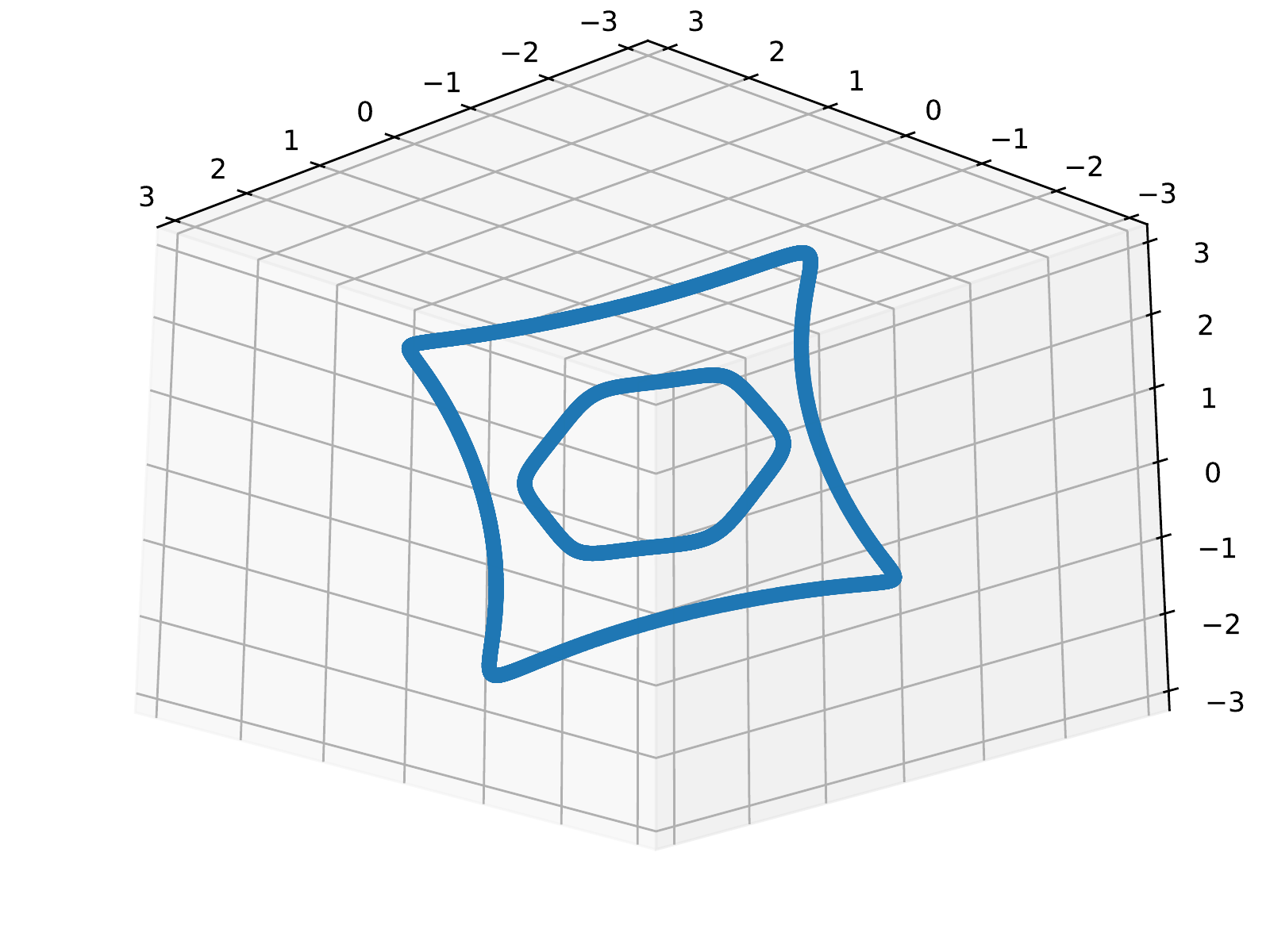}
\end{figure}
\end{ex}

\section{Efficiency} \label{sec:complexity}
The main complexity measure when comparing homotopy methods is the
number of paths. We will compare our approach to solving
(\ref{eq:homocomplex}) with $s=0$ using standard homotopies. For
example one may use a total degree homotopy but since the system is
linear in $v$ and $w$ the comparison was instead made to the more
efficient formulation as a multihomogeneous homotopy over $\CC^{2n}
\times \PP^a \times \PP^b$. There are many alternatives to this, for
example one may eliminate $x$ or $y$ from the equations
(\ref{eq:homocomplex}) with $s=0$, which will decrease the number of
variables but increase the degrees. Another alternative is to not use
the auxiliary variables $v$ and $w$ at all and express the
orthogonality conditions in terms of the vanishing of minors of
matrices composed from the Jacobians $J_F$, $J_G$ and the vector
$x-y$. If these reformulations are useful depends among other things
on the codimension of $X$ and $Y$.

As a first example consider general surfaces $X,Y \subset \CC^3$ of
degree $d$. In this case we have that $|I(X,Y)_0|$ is equal to
$(d^3-d^2+d)^2$ minus a contribution from $X\cap Y$ as in
\exref{ex:surfaces}. Below we include random examples of this kind for
some values of $d$, including one example where $X=Y$ and $X$ is
random of degree 3.

The beneficial effects of the specialized homotopy are best seen when
the invariants of $X$ and $Y$ are such that the number of solutions to
the problem is small. This is gauged during the start point
computations and fewer paths may be followed for the main homotopy in
these cases. For example, if $X,Y$ satisfy the assumptions
(\ref{eq:assumptions}) and $\bar{X},\bar{Y}$ are a smooth curves of
degree $d$ and genus $g$, $\EDD{X}=\EDD{Y}=3d+2(g-1)$. This means that
$\deg{I(X,Y)}=(3d+2(g-1))^2$ for general enough such curves.

As a first example of this consider general rational normal curves
$X,Y \subset \CC^n$ defined by the $2 \times 2$-minors of a $2\times
n$-matrix of general polynomials of degree one. Then
$\deg{I(X,Y)}=(3n-2)^2$. In contrast, the B\'{e}zout number of the
system (\ref{eq:homocomplex}) at $s=0$ is $2^{4n-2}$ and the
multihomogeneous root count is $n^22^{2n-2}$. We see that the start
point computation will actually dominate the complexity in this
situation since the multihomogeneous root count for the start points
is $n2^{n-1}$ while $\deg{I(X,Y)}$ is polynomial of degree 2 in $n$. Below
we include some random examples of this kind for various choices of
$n$.

As a further example consider curves of genus 1. A simple way to
generate such a curve is to take a smooth cubic $C \subset \PP^2$ and
embed $C$ in $\PP^n$ with $n={d+2 \choose 2}-1$ via the $d$-uple
Veronese embedding $v_d:\PP^2 \rightarrow \PP^n$. To get a curve $X
\subset \CC^n$ one may intersect with a general affine open $\CC^n
\subset \PP^n$ and choose coordinates on $\CC^n$ by eliminating one of
the homogeneous coordinates on $\PP^n$. For the examples we have
chosen $C \subset \PP^2$ defined by $x^3+y^3+z^3+xyz=0$. We repeat
this procedure twice to generate two distinct curves $X,Y \subset
\CC^n$ and run the algorithm to compute $I(X,Y)$, that is we use two
different random affine open subsets of $\PP^n$ to generate the curves
but then consider them as distinct curves in the same space
$\CC^n$. The curves $\bar{X}$ and $\bar{Y}$ are elliptic of degree
$3d$ and if the open affine subsets above are general, then
$\deg{I(X,Y)}=9^2d^2$. The ideals of $X$ and $Y$ are generated in
degree 2 and the multihomogeneous root count for these examples is the
same as for the rational curve examples, that is $n^22^{2n-2}$ with
$n={d+2 \choose 2}-1$.

\tabref{tab:paths} and \tabref{tab:timings} show the result in terms
of the number of paths and the time of the computation. The number of
solutions, that is the number of isolated points of $I(X,Y)$, is also
reported. In the case where $X=Y$ is a cubic surface, the number of
isolated bottlenecks is reported as the number of solutions. Cases
that took longer than 5 hours are marked with \mbox{"-"}. The examples
were run using \cite{M2bertini, bertini, M2} and a 2.8 GHz processor
of type Intel i7-2640M. In the tables, the homotopy suggested in this
paper is called "EDD" and the multihomogeneous homotopy is called
"multihom".

\begin{table}[ht] \caption{Some benchmarks: \#paths.}
  \begin{center}
  \label{tab:paths}
  \begin{tabular}{lllll}
    \hline
    Example & EDD & multihom & \#solutions \\
    \hline\\
    Two quadratic surfaces in $\CC^3$ & $2\cdot 6+36$ & 36 & 24\\
    Two cubic surfaces in $\CC^3$ & $2\cdot 36+441$ & 1296 & 396\\
    One cubic surface in $\CC^3$ & $36+210$ & 1296 & 138 \\
    Two quartic surfaces in $\CC^3$ & $2\cdot 108+2704$ & 11664 & 2592\\
    Two rational normal curves in $\CC^3$ & $2\cdot 12+49$ & 144 & 49\\
    Two rational normal curves in $\CC^4$ & $2\cdot 32+100$ & 1024 & 100\\
    Two rational normal curves in $\CC^5$ & $2\cdot 80+169$ & 6400 & 169\\
    Two rational normal curves in $\CC^6$ & $2\cdot 192+256$ & 36864 & 256\\
    Two rational normal curves in $\CC^7$ & $2\cdot 448+361$ & 200704 & 361\\
    Two rational normal curves in $\CC^8$ & $2\cdot 1024+484$ & 1048576 & 484\\
    Two elliptic curves in $\CC^5$ & $2\cdot 80+324$ & 6400 & 324\\
    Two elliptic curves in $\CC^9$ & $2 \cdot 2304+729$ & 5308416 & 729
  \end{tabular}
  \end{center}
\end{table}
\begin{table}[ht] \caption{Some benchmarks: time (s).}
  \begin{center}
  \label{tab:timings}
  \begin{tabular}{llll}
    \hline
    Example & EDD & multihom\\
    \hline\\
    Two quadratic surfaces in $\CC^3$ & 0.9 & 0.8\\
    Two cubic surfaces in $\CC^3$ & 6.3 & 18.5\\
    One cubic surface in $\CC^3$ & 6.4 & 24.1\\
    Two quartic surfaces in $\CC^3$ & 91.4 & 591.5\\
    Two rational normal curves in $\CC^3$ & 1.2 & 9.0\\
    Two rational normal curves in $\CC^4$ & 4.5 & 120.7\\
    Two rational normal curves in $\CC^5$ & 12.3 & 1462.5\\
    Two rational normal curves in $\CC^6$ & 31.1 & 16086.9\\
    Two rational normal curves in $\CC^7$ & 124.5 & -\\
    Two rational normal curves in $\CC^8$ & 347.7 & -\\
    Two elliptic curves in $\CC^5$ & 23.9 & 2183.9\\
    Two elliptic curves in $\CC^9$ & 867.1 & -
  \end{tabular}
  \end{center}
\end{table}

We conclude with a remark on how to measure the complexity of
computations in algebraic geometry. The number of solutions to a
problem such as the one studied in this paper depends on invariants of
algebraic varieties, such as Chern classes. As a complement to other
complexity analyses, it is useful to express the complexity of an
algorithm or a problem in terms of these invariants rather than the
length of the input system, the number of variables, the degrees of
defining equations and so on.

\end{document}